\documentclass[12pt, reqno]{amsart}



\title{Is it true that most sets are Sidon?}

\usepackage[T1]{fontenc}
\usepackage{amsmath}
\usepackage{amssymb}
\usepackage{amsthm}
\usepackage[left=3.5cm, right=3.5cm]{geometry} %ORIGINAL
\usepackage{hyperref}
\usepackage{fancyhdr}
\usepackage{enumitem}
\usepackage{comment}
\usepackage{nicefrac}
\usepackage{comment}
\usepackage{bm}
\usepackage{mathrsfs}
\usepackage{graphicx}
\usepackage[utf8]{inputenc}
\usepackage{cancel}
\usepackage{mathtools}

\usepackage{tikz}
\usetikzlibrary{decorations.pathreplacing,matrix,arrows,positioning,automata,shapes,shadows,calc,fadings,decorations,snakes,through,intersections}

\author[P.~Leonetti]{Paolo Leonetti}
\address{Department of Economics, Universit\`a degli Studi dell'Insubria, via Monte Generoso 71, 21100 Varese, Italy}
\email{leonetti.paolo@gmail.com}
\urladdr{\url{https://sites.google.com/site/leonettipaolo/}} 

\AtBeginDocument{%
   \def\MR#1{}
}

\newtheorem{thm}{Theorem}[section]
\newtheorem{prop}[thm]{Proposition}

\theoremstyle{definition} 
\newtheorem{defi}[thm]{Definition}%[section]
\let\olddefi\defi
\renewcommand{\defi}{\olddefi\normalfont}
\newtheorem{question}{Question}%[section]
\let\oldquestion\question
\renewcommand{\question}{\oldquestion\normalfont}
\newtheorem{example}[thm]{Example}
\let\oldexample\example
\renewcommand{\example}{\oldexample\normalfont}
\newtheorem{rmk}[thm]{Remark}
\let\oldrmk\rmk
\renewcommand{\rmk}{\oldrmk\normalfont}

\hypersetup{
    pdftitle={},
    pdfauthor={},
    pdfmenubar=false,
    pdffitwindow=true,
    pdfstartview=FitH,
    colorlinks=true,
    linkcolor=blue,
    citecolor=green,
    urlcolor=cyan
}

\providecommand{\MR}[1]{}

\providecommand{\MR}{\relax\ifhmode\unskip\space\fi MR }

\begin{document}

\maketitle
\thispagestyle{empty}

\begin{abstract} 
No.%%%%%%%%%%%
\end{abstract}

%%%%%%%%%%%%%%%%%%%%%%%%%%%%%%%%%%%%%%%%%%%%%%%%%%%%%%%%%%%%%%%%%%%%%%%%%%%%

\section{Introduction} 

A subset $A$ of the nonnegative integers $\mathbb{N}$ is said to be \emph{Sidon set} if for every $x \in \mathbb{N}$ the equation 
$$x=a+b$$ has at most one solution $(a,b) \in A^2$ with $a\le b$. 
The aim of this note is to show that, from a topological point of view, most subsets of $\mathbb{N}$ are not Sidon sets. 

%Given an integer $h\ge 2$, 
%A subset $A$ of an abelian semigroup $G$ is said to be a $B_h$\emph{-set} if for every $x \in G$ the equation $x=a+b$ has at most one pair of solutions in $A$

More generally, given an abelian semigroup $S$, an element $x \in S$, and an integer $h\ge 2$, we denote by $r_{A,h}(x)$ the number of solutions $(a_1,\ldots,a_h) \in A^h$ of the equation $x=a_1+\cdots+a_h$, 
where we consider two such solutions to be the same if they differ only in the ordering of the summands. 
%counted up to permutations of the variables $a_i$. 
%given an integer $h\ge 2$, a subset $A$ of an abelian semigroup $G$ is said to be a $B_h$\emph{-set} if for every $x \in G$ the equation $x=a_1+\cdots+a_h$ has at most one solution $(a_1,\ldots,a_h) \in A^h$, up to permutations of the variables $a_i$. 
%Of course, if $G=\mathbb{N}$ and $h=2$, then $B_2$-set are the classical Sidon sets. 
It is worth to remark that, given an integer $g\ge 1$, sets $A\subseteq S$ which satisfy 
$$
\forall x \in S, \quad r_{A,h}(x) \le g
$$
are usually called $B_h[g]$\emph{ sets}; if $g=1$, they are simply called $B_h$\emph{ sets}. We denote by $\mathcal{B}_{h,g}$ the family of all $B_h[g]$ sets. 
Of course, in the case $S=\mathbb{N}$, a set $A\subseteq \mathbb{N}$ is Sidon if and only if it is a $B_2$ set (i.e., $r_{A,2}(n) \le 1$ for all $n \in \mathbb{N}$); we refer to \cite{MR2680183, Mel25} and references contained therein for related literature about $B_h[g]$ sets. 

In our first main result, we show, from a topological viewpoint, that most subsets of $\mathbb{N}$ are not $B_h[g]$. For, we identify $\mathcal{P}(\mathbb{N})$ with the Cantor space $\{0,1\}^{\mathbb{N}}$. In particular, we can speak about the topological complexity of subsets of $\mathcal{P}(\mathbb{N})$. 
\begin{thm}\label{thm:main}
%Fix integers $h\ge 2$ and $g\ge 1$. Then 
Each 
$\mathcal{B}_{h,g}$ 
is closed and with empty interior \textup{(}hence, meager\textup{)}. 
%In particular, the family of $B_{h}[g]$ subsets of $\mathbb{N}$ is closed and with empty interior \textup{(}hence, meager\textup{)}. 
%\textcolor{red}{[Complete]}
\end{thm}

In particular, considering the family of meager sets is a $\sigma$-ideal, it follows by Theorem \ref{thm:main} that the family of sets $A\subseteq \mathbb{N}$ which are $B_h[g]$ sets for some $h\ge 2$ and $g\ge 1$ is a meager subset of $\mathcal{P}(\mathbb{N})$ (of course, this is meaningful since the $\mathcal{P}(\mathbb{N})$ is a Polish space, hence it is not meager in itself). 

We are going to strengthen the meagerness claim given in Theorem \ref{thm:main} in two directions. To this aim, recall that 
\begin{equation}\label{eq:erdos}
r_{\{0,1,\ldots,n\},\,h}(n) \sim \frac{n^{h-1}}{h!(h-1)!} \quad \text{ as }n\to \infty
\end{equation}
for all integers $h\ge 2$; see \cite[Section 4]{MR4841}, and cf. also \cite[Theorem 4.2.1]{MR2260521} and \cite{MR1705753}.
%connection with Nath: https://math.stackexchange.com/questions/4421715/proof-the-number-of-partitions-of-n-into-at-most-m-parts-is-the-number-of-part
%https://oeis.org/search?q=0%2C0%2C1%2C1%2C2%2C3%2C4%2C5&language=italian&go=cerca 
%THEORY OF PARTITION ANDREWS 76.
First, we show that the behavior of $r_{A,h}$ is really wild for most sets $A$. Taking into account \eqref{eq:erdos}, the next result is optimal: 
\begin{thm}\label{thm:mainA}
    Fix an integer $h\ge 2$ and pick a divergent map $f: \mathbb{N}\to \mathbb{N}$ such that $f(n)=o(n^{h-1})$ as $n\to \infty$. 
    Then the family of all sets $A\subseteq \mathbb{N}$ such that 
    $$
    \liminf_{n\to \infty}r_{A,h}(n)=0 \quad \text{ and }\quad \limsup_{n\to \infty}\frac{r_{A,h}(n)}{f(n)}=\infty
    $$
    is 
    comeager. 
    %open and dense \textup{(}hence, comeager\textup{)}. 
\end{thm}

Second, we improve the meagerness claim in Theorem \ref{thm:main} in a different sense: informally, most subsets of $\mathbb{N}$ and all their \textquotedblleft small pertubations\textquotedblright\, are not $B_h[g]$ sets. 
This goes in the analogue direction of \cite{MR4607615}. 
For, we denote by $\mathsf{d}_\star(A)$ the lower asymptotic density of a subset $A\subseteq \mathbb{N}$, that is, 
$$
\mathsf{d}_\star(A):=\liminf_{n\to \infty}\frac{|A\cap [0,n)|}{n}, 
$$
see e.g. \cite{MR4054777}. 
%In addition, we identify $\mathcal{P}(\mathbb{N})$ with the Cantor space $\{0,1\}^{\mathbb{N}}$. In particular, we can speak about the topological complexity of subsets of $\mathcal{P}(\mathbb{N})$. 
\begin{thm}\label{thm:mainB}
%Fix integers $h\ge 2$ and $g\ge 1$. Then 
%$\mathcal{B}_{h,g}$ 
The family of subsets $A\subseteq \mathbb{N}$ which admit, for some integers  $h\ge 2$ and $g\ge 1$, a $B_h[g]$ set $B\subseteq \mathbb{N}$ such that $\mathsf{d}_\star(A\bigtriangleup B)<1$  
%$B$ is not a $B_h[g]$ set whenever $B\subseteq \mathbb{N}$ satisfies $\mathsf{d}_\star(A\bigtriangleup B)<1$ 
is meager.  % \textcolor{red}{[Complete]}
%In particular, the family of $B_{h}[g]$ subsets of $\mathbb{N}$ is closed and with empty interior \textup{(}hence, meager\textup{)}. 
%\textcolor{red}{[Complete]}
\end{thm}
It is worth to remark that, differently from Theorem \ref{thm:main}, Theorem \ref{thm:mainA} and Theorem \ref{thm:mainB} do not provide the topological complexities of the claimed sets. 

Lastly, on the opposite direction of all the above results, we conclude by showing that the family of $B_h[g]$ sets is comeager in the realm of finite subsets with given cardinality. This extends a recent result of 1 \cite{Mel25} (which coincides with the case $g=1$ and $X$ finite dimensional) and answers the open question contained therein. 
\begin{prop}\label{prop:nathanson}
    Let $X$ be a \textup{(}real or complex\textup{)} topological vector space and fix integers $n,h \ge 2$ and $g\ge 1$. Then the family of all $B_h[g]$ subsets $A\subseteq X$ with $|A|=n$ can be regarded as an open dense 
    \textup{(}hence, comeager\textup{)} 
    subset of $X^n$. 
\end{prop}

Thus, 
it follows by Proposition \ref{prop:nathanson} that 
the family of all $A\subseteq X$ with cardinality $n$ which are $B_h[g]$ sets for every $h\ge 2$ and $g\ge 1$ can be regarded as a dense $G_\delta$ subset of $X^n$, hence comeager. All proofs are given in the next section.

\section{Proofs}

\begin{proof}
[Proof of Theorem \ref{thm:main}]
    Since $\mathcal{P}(\mathbb{N})$ is identified with $\{0,1\}^{\mathbb{N}}$, 
a basic open set in $\mathcal{P}(\mathbb{N})$ will be a cylinder of the type $\{A\subseteq \mathbb{N}: A \cap [0,k]=F\}$ for some integer $k \in \mathbb{N}$ and some (possibly empty) finite set $F\subseteq [0,k]$.

First, let us show that $\mathcal{B}_{h,g}$ is closed. This is obvious is $\mathcal{B}_{h,g}=\mathcal{P}(\mathbb{N})$. Otherwise, pick a set $A_0\notin \mathcal{B}_{h,g}$. Hence, there exists $n_0 \in \mathbb{N}$ such that $r_{A_0,h}(n_0)\ge g+1$. At this point, define the open set 
    $$
    \mathcal{U}:=\left\{A\subseteq \mathbb{N}: A \cap [0,n_0]=A_0 \cap [0,n_0]\right\}. 
    $$
    It would be enough to note that $\mathcal{U} \cap \mathcal{B}_{h,g}=\emptyset$: in fact, if $A \in \mathcal{U}$, then $r_{A,h}(n_0)=r_{A_0,h}(n_0)\ge g+1$, hence $A$ is not a $B_{h,g}$ set. 

    Lastly, we need to show that $\mathcal{B}_{h,g}$ has empty interior. In fact, suppose that there exists a nonempty finite set $F_0\subseteq \mathbb{N}$ such that each $A\subseteq\mathbb{N}$ is a $B_h[g]$ set whenever $F_0\subseteq A$. 
    At this point, define the finite set 
    $$
    A:=F_0\cup \{x_0,x_0+1,\ldots,h(x_0+g)\}, 
    $$
    where $x_0:=1+\max F_0$. Note that, for each $i \in \{0,1,\ldots,g\}$, we have 
    $$
    h(x_0+g)=(h-2)(x_0+g)+(x_0+g+i)+(x_0+g-i), 
    $$
    and that all the above variables belong to $A$. This provides at least $g+1$ distinct representations of the integer $h(x_0+g)$, that is, $r_{A,h}(h(x_0+g))\ge g+1$. Therefore 
    %$r_{A,h}(h(x_0+g))\ge g+1$, so that 
    $A$ is not a $B_h[g]$ set, completing the proof. 
\end{proof}

\bigskip

\begin{proof}
    [Proof of Theorem \ref{thm:mainA}]
    Let $\mathcal{S}$ be the claimed set. We are going to use the Banach--Mazur game defined as follows, see \cite[Theorem 8.33]{K}\textup{:} 
Players I and II choose alternatively nonempty open subsets of $\{0,1\}^{\mathbb{N}}$ as a nonincreasing chain 
$
\mathcal{U}_0\supseteq \mathcal{V}_0 \supseteq \mathcal{U}_1 \supseteq \mathcal{V}_1\supseteq \cdots, 
$ 
where Player I chooses the sets $\mathcal{U}_0,\mathcal{U}_1,\ldots$; Player II is declared to be the winner of the game if 
\begin{equation}\label{eq:claimgame}
\bigcap\nolimits_{m\ge 0} \mathcal{V}_m \cap \, \mathcal{S}\neq \emptyset. 
\end{equation}
Then Player II has a winning strategy (that is, he is always able to choose suitable sets $\mathcal{V}_0, \mathcal{V}_1,\ldots$ so that \eqref{eq:claimgame} holds at the end of the game) if and only if $\mathcal{S}$ is a comeager set in the Cantor space $\{0,1\}^{\mathbb{N}}$. 

At this point, we define define the strategy of player II recursively as it follows. Suppose that the nonempty open sets $\mathcal{U}_0\supseteq \mathcal{V}_0 \supseteq \cdots \supseteq \mathcal{U}_{m}$ have been already chosen, for some $m \in \mathbb{N}$. Then there exists a nonempty finite set $F_m \subseteq \mathbb{N}$ and an integer $k_m \ge \max F_m$ such that 
$$
\forall A\subseteq \mathbb{N}, \quad 
A\cap [0,k_m]=F_m \implies A \in \mathcal{U}_m.
$$
%Define the finite set $B_m:=\mathbb{N} \cap [k_m+1, h(1+k_m)]$ and 
Set $x_m:=h(1+k_m)+1$ and let $t_m$ be an integer such that $t_m\ge x_m+1$ and 
$$
%t_m>h(1+k_m)\quad \text{ and }\quad 
r_{\{x_m, x_m+1, \ldots,t_m\},\,h}(t_m) \ge mf(t_m).
$$
Note that this integer $t_m$ actually exists since it follows by \eqref{eq:erdos} that
$$
r_{\{x_m, x_m+1, \ldots,n\},\,h}(n)=r_{\{0,1,\ldots,n-hx_m\},h}(n-hx_m)
\sim \frac{n^{h-1}}{h!(h-1)!}
$$
as $n\to \infty$. Lastly, define the open set 
$$
\mathcal{V}_m:=\left\{A\subseteq \mathbb{N}: A \cap [0,t_m]=F_m \cup \{x_m,x_m+1,\ldots,t_m\}\right\}. 
$$

To complete the proof, it is enough to show that this is indeed a winning for Player II. In fact, pick $A \in \bigcap_m \mathcal{V}_m$ (which is possible). It follows that, for each $m \in \mathbb{N}$, the equation $h(1+k_m)=a_1+\cdots+a_h$ with $a_1,\ldots,a_h \in A$ does not have solutions since $x_m> \max\{a_1,\ldots,a_h\} \ge h(1+k_m)/h>k_m$ and, by construction, $(k_m,x_m) \cap A=\emptyset$. It follows that 
$$%\begin{equation}\label{eq:claimineq1}
%\forall m \in \mathbb{N}, \quad 
r_{A,h}(h(1+k_m))=0 
$$%\end{equation}
for all $m \in \mathbb{N}$, hence $\liminf_n r_{A,h}(n)=0$. 
In addition, by construction we have
$$%\begin{equation}\label{eq:claimineq2}
r_{A,h}(t_m)\ge r_{A\cap [x_m,t_m]} \ge mf(t_m)
$$%\end{equation}
for all $m\in \mathbb{N}$, so that $\limsup_n r_{A,h}(n)/f(n)=+\infty$. Therefore $A \in \mathcal{S}$. 
%\textcolor{red}{[Complete]}
\end{proof}

\bigskip

\begin{proof}
[Proof of Theorem \ref{thm:mainB}]
    Fix integers $h\ge 2$ and $g\ge 1$, and let $\mathcal{M}$ be the family of subsets $A\subseteq \mathbb{N}$ such that, for all sets $B\subseteq \mathbb{N}$ such that $\mathsf{d}_\star(A\bigtriangleup B)<1$, $B$ is not a $B_h[g]$ set.   %   which admit a $B_h[g]$ set $B\subseteq \mathbb{N}$ such that $\mathsf{d}_\star(A\bigtriangleup B)<1$. 
    It is enough to show that $\mathcal{M}$ is comeager. For, we are going to use the same strategy used in the proof of Theorem \ref{thm:mainA}. 
    
    The strategy of Player II is defined as follows: suppose that the open set $\mathcal{U}_m$ has been chosen, and define $F_m$ and $k_m$ as in the proof of Theorem \ref{thm:mainA}. 
    Pick an integer $y_m>mk_m$ and     
    %Pick a sufficiently large integer $y_m$ such that 
    %$$
    %y_m >\max\{mk_m, 2m, (h\cdot h!)^{1-h}\}
    %$$ 
    %and 
    define 
    $$
    \mathcal{V}_m:=\left\{A\subseteq \mathbb{N}: A \cap [0,y_m]=F_m\cup \{k_m+1,k_m+2,\ldots,y_m\}\right\}.
    $$

    To conclude the proof, it is enough to show this is indeed a winning strategy for Player II. To this aim, pick $A \in \bigcap_m \mathcal{V}_m$ and fix $B\subseteq \mathbb{N}$ such that 
    $$
    \alpha:=\mathsf{d}_\star(A\bigtriangleup B)\in [0,1).
    $$
    We claim that $B$ is not a $B_h[g]$ set. For, suppose by contradiction that $r_{B,h}(n)\le g$ for all $n \in \mathbb{N}$. Pick a positive integer $m_0$ such that $\alpha<1-1/m_0$ and $|(A\bigtriangleup B) \cap [1,n]|\le n(1-1/m_0)$ for all $n\ge m_0$. Observe also that $y_m> mk_m$ implies 
    $$
    %\frac{|A\cap [0,y_m+1)|}{y_m+1}\ge 
    \frac{|A\cap [1,y_m]|}{y_m}\ge 
    \frac{|A\cap (k_m,y_m]|}{y_m} \ge 1-\frac{1}{m+1} %\ge 1-\frac{1}{m}
    $$
    for all $m\in \mathbb{N}$. Setting $c:=1/m_0-1/(m_0+1)>0$, it follows that 
    \begin{displaymath}
        \begin{split}
    \frac{|B\cap (k_m,y_m]|}{y_m}&\ge 
    \frac{|A\cap (k_m,y_m]|-|(A\bigtriangleup B) \cap [1,n]|}{y_m} \\
    &\ge 1-\frac{1}{m+1}-\left(1-\frac{1}{m_0}\right)\ge c
    \end{split}
    \end{displaymath}
    for all $m\ge m_0$. Since that there are at least $cy_m$ integers in $B \cap (k_m,y_m]$ and all sums of $h$ distinct terms from $B \cap [0,y_m]$ are upper bounded by $hy_m$, we obtain 
    $$
    \binom{\lceil cy_m\rceil }{h} \le \sum_{n=1}^{hy_m}r_{B,h}(n) \le hg y_m 
    $$ 
    for all $m\ge m_0$. However, the latter inequality fails if $m$ is sufficiently large since the left hand side is $\gg y_m^h$ (and $h\ge 2$). This concludes the proof. 
    %It follows that, if $m\ge m_0$ is sufficiently large, then
    %$$
    %\frac{|A\cap [0,y_m+1)|}{y_m+1}\ge \frac{|A\cap [1,y_m]|}{y_m+1}\ge 1-\frac{1}{m+2}.
    %$$
\end{proof}

\bigskip

\begin{proof}
[Proof of Proposition \ref{prop:nathanson}]
    Identify each set $A=\{a_1,\ldots,a_n\}\subseteq X$ with the vector $(a_1,\ldots,a_n) \in X^n$. Hence, we need to show that the set $\mathcal{V}$ of vectors $a \in X^n$ with pairwise distinct coordinates such that $\{a_1,\ldots,a_n\}$ is a $B_h[g]$ set is open dense in $X^n$. 
    For, note that $a\notin \mathcal{V}$ if and only if either $a_i=a_j$ for some $i\neq j$ 
    or there exist pairwise distinct vectors $\alpha_1,\ldots,\alpha_{g+1} \in \mathbb{N}^n$ such that 
    $$
    \sum_{i=1}^n\alpha_{1,i}=\cdots=\sum_{i=1}^n\alpha_{g+1,i}=h
    $$
    and 
    $$
    \sum_{i=1}^n\alpha_{1,i}a_i=\cdots=\sum_{i=1}^n\alpha_{g+1,i}a_i.
    $$
    %$\sum_i\alpha_i=\sum_i\beta_i=h$ and $\sum_i\alpha_ia_i=\sum_i\beta_ia_i$. 
    In both cases, $a$ belongs to a finite intersection of hyperplanes  $\mathcal{H}_\gamma\subseteq X^n$ of the type $\sum_i\gamma_ix_i=0$ for some $\gamma \in \Gamma:=(\mathbb{Z} \cap [-h,h])^n\setminus \{0\}$. 
    Notice that $\Gamma$ is finite, and that each $\mathcal{H}_\gamma$ is the kernel of the nonzero continuous linear operator $T_\gamma: X^n\to X$ defined by $T_\gamma(x):=\sum_i\gamma_ix_i$, hence $\mathcal{H}_\gamma$ is closed and with empty interior, see e.g. \cite[Propositions 1.9.6 and 1.9.9]{MR3616849}. It follows that 
    $$
    X^n\setminus \mathcal{V}
    $$
    can be rewritten as a finite union of finite intersections of hyperplanes $\mathcal{H}_\gamma$. Thus $X^n\setminus \mathcal{V}$ is closed with empty interior, which is equivalent to our claim. 
    %Thus 
    %$$
    %X^n\setminus \mathcal{B}_h\subseteq \bigcup\nolimits_{\gamma \in \Gamma}\mathcal{H}_\gamma.
    %$$
    %The claim follows by the fact that $\Gamma$ is finite and each $\mathcal{H}_\gamma$ is the kernel of the nonzero continuous linear operator $T_\gamma: X^n\to X$ defined by $T_\gamma(x):=\sum_i\gamma_ix_i$, hence $\mathcal{H}_\gamma$ is closed and with empty interior, see e.g. \cite[Propositions 1.9.6 and 1.9.9]{MR3616849}. Therefore $\mathcal{B}_h$ is open and dense. % (and, in particular, comeager). 
\end{proof}

%\nocite{*}
\bibliographystyle{amsplain}
\bibliography{idealezzz}

\end{document}